\documentclass[submission]{dmtcs-episciences}

\usepackage[utf8]{inputenc}
\usepackage{subfigure}
\usepackage{lineno,hyperref}
\modulolinenumbers[5]
\usepackage{amsmath}
\usepackage{graphicx}
\usepackage{colortbl}
\usepackage{longtable}
\usepackage{epsf}
\usepackage{color}

\usepackage{amsthm}

\newtheorem{theorem}{Theorem}
\newtheorem{lemma}{Lemma}

\newtheorem{definition}{Definition}

\usepackage{amssymb}
\newcommand{\UCS}{{Union Closed Set }}
\newcommand{\UCSs}{{Union Closed Sets }}
\newcommand{\UCSp}{{Union Closed Set}}
\newcommand{\UCSsp}{{Union Closed Sets}}
\newcommand{\red}{{\mbox{red}}}

\author{Gunnar~Brinkmann\thanks{Gunnar.Brinkmann@UGent.be} \and
Robin Deklerck\thanks{Robin.Deklerck@UGent.be}}
\affiliation{Universiteit Gent, Toegepaste wiskunde, informatica en statistiek, Krijgslaan 281 S9, \\B 9000 Gent, Belgium}

\title{Generation of Union Closed Sets and \\ Moore families}

\begin{document}

\maketitle
\begin{abstract}

In this article we will describe an algorithm to 
constructively enumerate 
non-isomorphic \UCSs and Moore sets. We
confirm the number of isomorphism classes of \UCSs and Moore sets
on $n\le 6$ elements presented by other authors and give the 
number of isomorphism classes of \UCSs and Moore sets
on $7$ elements. 
Due to the enormous growth of
the number of isomorphism classes it seems unlikely that
constructive enumeration for $8$ or more elements will be possible
in the foreseeable future. 

\end{abstract}

\section*{\underline{Introduction}}

All sets in this article are finite.  A \UCS is a set ${\cal U}$ of sets
with the property that for all $A,B\in {\cal U}$ we have that $A\cup
B\in {\cal U}$. We call $\Omega_{{\cal U}}=\bigcup_{A\in {\cal U}}A$ the universe
of ${\cal U}$. Two \UCSs with universe $\Omega_{{\cal U}}$, resp. $\Omega_{{\cal U}'}$
are defined to be isomorphic if there is a bijective mapping 
$\Omega_{{\cal U}} \to \Omega_{{\cal U}'}$ inducing a bijection between the
\UCSsp.

As we are only interested in isomorphism classes, we may assume
$\Omega_U= \Omega_n=\{1,\dots ,n\}$ for some $n$.  While the whole
universe
$\Omega_U$ is by definition an element of a \UCS ${\cal U}$,
this is not the case for the empty set.
As the empty set has no
impact on the structure of a \UCSp, one often either requires the
empty set to be an element of each \UCS or forbids it to be an
element. We choose for the first convention, so our \UCSs contain $\Omega_n$
as well as the empty set. We denote a set containing one representative
of each isomorphism class of \UCSs with universe  $\Omega_n$ as ${\cal R}_n$.

Although the famous \UCSs conjecture (or Frankl's conjecture) is exactly about the structures 
we generate here, our approach is not suitable for testing this conjecture.
A lot is known about the structure of possible counterexamples to 
the \UCSs conjecture (see \cite{UCSsurvey} for a survey), so any 
approach to extend the knowledge on 
the smallest size of a possible counterexample by constructive enumeration
must focus on the subset of \UCSs with those additional structural properties
(e.g. with small average size of the sets, without some subconfigurations 
like singletons, etc.).

\UCSs are closely related to {\em Moore Families}. A Moore family
for a universe $\Omega_n$ is a set  of subsets of $\Omega_n$
that is closed under intersection and contains $\Omega_n$. It is easy to
see that for a \UCS ${\cal U}$ the set ${\cal U}^c=\{\Omega_n\setminus A | A \in {\cal U}\}$
is a Moore family. For a Moore family ${\cal M}$ the set ${\cal M}^c=\{\Omega_n\setminus A | A \in {\cal M}\}$
is closed under union, but as the empty set is not necessarily contained in ${\cal M}$, it is a
\UCS for $\Omega_n\setminus \bigcap_{A\in \cal M} A$, which is isomorphic to a \UCS for some
$\Omega_{n'}$ with $n'\le n$.

A set ${\cal M}_n$ of representatives of Moore families (with the canonical definition
of isomorphism) for the universe $\Omega_n$ can be obtained from sets
$ {\cal R}_0, \dots , {\cal R}_n$ of representatives of \UCSs containing the empty set
as 

${\cal M}_n=\bigcup_{i=0}^n \{ {\cal U}^c  |  {\cal U}\in {\cal R}_i\}$

if the complement is in each case taken in the universe $\Omega_n$.

In \cite{Moorefamilies5}, \cite{Moorefamilies6} and \cite{Moorefamilies7} 
Moore families are enumerated. 
In the most advanced of these articles -- \cite{Moorefamilies7} --  all
Moore families for $n\le 6$ were counted and representatives of 
isomorphism classes were generated. For $n=7$ the approach is not suitable for
generating a set of representatives and only the number of labeled Moore families
-- that is without considering isomorphisms -- was determined 
by clever use of the structure
of representatives of Moore families for $n=6$. In our algorithm we determine the number of
labeled \UCSs resp. labeled Moore families for $n= 7$ from representatives and 
their automorphism groups of the \UCSs for $n= 7$, resp. $n\le 7$.
This gives a very good independent test
for the implemetation in \cite{Moorefamilies7} as well as for our implementation.
When computing the number of labeled Moore families for $\Omega_7$ from the
number of labeled \UCSs with $n\le 7$, for those \UCSs with $n< 7$, 
the possible ways to choose $n$ vertices
from $\{1,\dots ,7\}$ are also taken into account.

\section*{\underline{The algorithm}}

A subset $A \subseteq \Omega_n$ is represented as a number 
$b(A)$ given as the binary number $b_{n-1}\dots b_0$ -- possibly with leading zeros -- with $b_i=1$
if $(i+1)\in A$ and $b_i=0$ otherwise.

We use an ordering of the subsets of $\Omega_n$. For $A,B \subseteq \Omega_n$
we define $A > B$ if $|A|<|B|$ (so sets with more elements are considered smaller in this order)
or $|A|=|B|$ and $b(A)>b(B)$. Whenever we refer to {\em larger} or {\em smaller} sets, we refer
to this ordering. 

The construction algorithm generates \UCSs recursively based on
the following easy lemma:

\begin{lemma}

If ${\cal U} \not= \{\Omega_n\}$ is a \UCS and $A$ is the largest 
non-empty element
in ${\cal U}$, then ${\cal U}\setminus \{A\}$ is also a \UCSp.

\end{lemma}

This implies that \UCSs for universe  $\Omega_n$ can be recursively constructed
from the \UCS $\{\Omega_n,\emptyset\}$ of smallest size by successively adding subsets of $\Omega_n$
that are larger than the largest non-empty set already present.
Of course it is not assured that adding a smaller set to a \UCS does not violate
the condition that the set must be closed under unions.

In order to turn this into an efficient algorithm, two tests that are 
(in principle) applied to each structure generated must be very fast:

\begin{description}
\item[(i)] The test whether the set that has been constructed by adding a new set
is closed under union.

\item[(ii)] The test for isomorphisms.
\end{description}

We will first discuss (i). A straightforward way to test (i) for a \UCS ${\cal U}$ to which a new
set $A$ is added would be to form all unions $A\cup B$ with $B\in {\cal U}$ and test whether 
they are in ${\cal U}\cup \{A\}$. Although all these steps can be implemented as 
efficient bit-operations, their number would slow down the program. We define:

\begin{definition}

For a \UCS ${\cal U}$ we define the reduced set $\red({\cal U})$ as

$\red({\cal U})=\{ A\in {\cal U} | A\not= \emptyset \mbox{ and there is no }A_1\not= \emptyset \mbox{ in } {\cal U}, A_1\subsetneq A\}$

\end{definition}

\begin{lemma}

Let ${\cal U}$ be a \UCS for a universe $\Omega_n$ and let $A\subset \Omega_n$,
that is larger than any non-empty set in ${\cal U}$.
Then ${\cal U}\cup \{A\}$ is closed under union if and only if
$A\cup B\in {\cal U}$ for all $B\in \red({\cal U})$.

\end{lemma}

\begin{proof}

Assume first that ${\cal U}\cup \{A\}$ is closed under union
and let $B\in \red({\cal U})$.
Then $A\cup B\in ({\cal U}\cup \{A\})$ and 
as $A$ is larger than $B$, we have $A\cup B \not= A$, so $A\cup B\in {\cal U}$.

For the other direction assume that 
$A\cup B\in {\cal U}$ for all $B\in \red({\cal U})$ and let 
$D\in {\cal U}$.

Choose any $D'\subset D$ so that $D'\in \red({\cal U})$. Then $A'=A\cup D'\in {\cal U}$
and therefore also $A'\cup D \in {\cal U}$ as ${\cal U}$ is closed under union, but 
$A'\cup D = A\cup D$ -- so $A\cup D\in {\cal U}\cup \{A\}$ and 
${\cal U}\cup \{A\}$ is closed under union.

\end{proof}

It would be inefficient to compute $\red({\cal U})$ each time a new \UCS is constructed,
but as a new \UCS  ${\cal U}'$ is constructed by adding a new smallest element $A$ to
${\cal U}$, the set  $\red({\cal U}')$ can easily be constructed from $\red({\cal U})$
by adding $A$ and removing elements that contain $A$. Nevertheless the few lines
of code testing whether the potential \UCS is closed under union take more than $50\%$
of the total running time when computing \UCSs for $\Omega_6$, which is the largest case that can be profiled.

In order to solve problem (ii) efficiently -- that is avoid the generation of isomorphic copies,
we use a combination of Read/Farad{\v z}ev type orderly generation 
\cite{Fa76}\cite{Read78} and the homomorphism principle (see e.g. \cite{B98}).

Our first aim is to define a unique representative for each isomorphism class
-- called the canonical representative -- 
and then only construct \UCSs that are the the canonical representatives of their class. We
represent a \UCS ${\cal U}$ with $k+1$ elements $A_1<A_2<\dots <A_k<\emptyset$ as the string 
$s({\cal U})=b(A_1),\dots ,b(A_k)$ of numbers.
For a given isomorphism class of \UCSs for a universe $\Omega_n$
we choose the \UCS with the lexicographically smallest string as the representative.

It is {\em in principle} easy to test whether a given \UCS ${\cal U}$ is the representative 
of its class by applying all
$n!$ possible permutations to ${\cal U}$ and comparing the strings. 
As $n\le 7$ this would not be extremely expensive, but due to the large number of
times that this test has to be applied, still too expensive to construct the \UCSs for
$\Omega_7$. In the sequel we will describe a way how this can be optimized.

In order to increase the efficiency by making it an orderly algorithm of Read/Farad{\v z}ev type, we will
use the canonicity test not only for structures we output, but also during the construction:
non-canonical structures are neither output nor used in the construction. This will only lead
to a correct algorithm if we can prove that canonical representatives are constructed
from canonical representatives:

\begin{theorem}\label{thm:orderly}

Let ${\cal U} \not= \{\Omega_n,\emptyset\}$ be a \UCS for the universe $\Omega_n$ that is the 
canonical representative for its isomorphism class. If ${\cal U}=\{A_1,A_2,\dots ,A_k,\emptyset\}$
with $A_1<A_2<\dots <A_k$ and $1\le m \le k$, then $\{A_1,A_2,\dots ,A_m,\emptyset\}$
is also the canonical representative of its class.

\end{theorem}

\begin{proof}
We prove the result for $m=k-1$. For $k=m$ it is the assumption and for 
$m<k-1$ it then follows by induction.

Let $s({\cal U})=(s_1,\dots ,s_{k})$. For a permutation $\Pi$ of $\Omega_n$ and a \UCS
${\cal U}$ we write $\Pi({\cal U})$ for the \UCS obtained by replacing each element
of a set in ${\cal U}$ by its image under $\Pi$. 
Assume that ${\cal U}'=\{A_1,A_2,\dots ,A_{k-1},\emptyset\}$ is not the canonical representative
of its class. So there is a permutation $\Pi$ of $\Omega_n$ with $s(\Pi({\cal U}'))<s({\cal U}')$.
Let $s(\Pi({\cal U}'))=(p_1, \dots ,p_{k-1})$ and we have $s({\cal U}')=(s_1,\dots ,s_{k-1})$.
Let $j$ be the first position so that $p_j<s_j$. 
Let us now look at $s(\Pi({\cal U}))=(p'_1,\dots ,p'_k)$ and let $r$ be the position
of $\Pi(A_k)$ in this string. If $r>j$ then $p'_i=p_i=s_i$ for $1\le i <j$ and
$p'_j=p_j<s_j$ -- so there is a smaller representative for the isomorphism class of
${\cal U}$. If  $r\le j$ then $p'_i=p_i=s_i$ for $1\le i <r$ and  
$p'_r<p_r\le s_r$ and again we have found a smaller representative contradicting the minimality of
$s({\cal U})$.

\end{proof}

This theorem proves that the algorithm can reject non-canonical \UCSs and is a correct
orderly algorithm, but the cost of the canonicity test would make it impossible to determine
the number of \UCSs for $\Omega_7$. 

For a given \UCS ${\cal U}$ with universe $\Omega_n$
and $1\le m \le n$ we write ${\cal U}_m$ for the subset containing all sets of
size $k\ge m$ and $\Pi_m({\cal U})$ for all permutations $\Pi$ of $\Omega_n$
with the property that $\Pi({\cal U}_m)={\cal U}_m$.


\begin{lemma}\label{lem:homomorph}

Let ${\cal U}\not=\{\Omega_n,\emptyset\}$ be a non-canonical \UCS 
for the universe $\Omega_n$ with sets $A_1<A_2<\dots <A_k<\emptyset$, so that
$\{A_1,A_2,\dots ,A_{k-1},\emptyset \}$ is canonical and $|A_k|=m$. Then all permutations
$\Pi$ with $s(\Pi({\cal U}))<s({\cal U})$ are in $\Pi_{m+1}({\cal U})$.

\end{lemma} 

\begin{proof}

Any permutation $\Pi$ not in $\Pi_{m+1}({\cal U})$ would by definition give
an isomorphic but different \UCS $\Pi({\cal U}_{m+1})$. As due to 
Theorem~\ref{thm:orderly}
$s({\cal U}_{m+1})$ is minimal, $s(\Pi({\cal U}_{m+1}))$ would be larger
and therefore also the part of the string of $s(\Pi({\cal U}))$ 
describing sets of size at least $m+1$ would imply $s(\Pi({\cal U}))>s({\cal U})$.

\end{proof}

Lemma~\ref{lem:homomorph} speeds up the canonicity test dramatically:
We start with a list of all $n!$ permutations as $\Pi_n({\cal U})$.
When testing canonicity of a \UCS with the smallest set of size $k<n$, we
only apply permutations from $\Pi_{k+1}({\cal U})$. During this application, we can 
already compute $\Pi_{k}({\cal U})$ by simply adding exactly those
permutations to the initially empty set  $\Pi_{k}({\cal U})$ that fix ${\cal U}$.
As we work with small sets, it is no problem to store and use the set of all
group elements instead of just a set of generators.

The impact is immediately clear: the number of permutations that has to be 
computed is much smaller and as soon as some $\Pi_{k}({\cal U})$ contains only the identity
-- which happens very fast -- no canonicity tests have to be performed, so that 
the total time for isomorphism checking is only about $7\%$
of the total running time when computing \UCSs for $\Omega_6$.

\subsection*{\underline{The implementation}}

The algorithm was implemented in $C$. Next to an efficient algorithm, of course also
implementation details are of crucial importance to be able to compute the results
for $\Omega_7$. We precomputed the action of all permutations on all sets, so that
they could be applied very fast and used data structures that allow to check
whether a set is contained in a \UCS in constant time. 
Special functions were written that add sets with
only one element. As no sets of smaller size will be added, no updates of the automorphism
groups are necessary and it turned out that at this stage it is also not efficient any more
to remove sets from $\red(\cal U)$ that are a superset of another element. Details
on the implementation can best be seen in the code which can be obtained from the 
authors.

\subsection*{\underline{Results}}

Tables 1 and 2 give the numbers of isomorphism classes of
\UCSs and Moore families as well as the numbers of labeled structures.
Up to $5$ elements the running times are less than $0.01$ seconds. 
For $n=6$
it is $8.2$ seconds on a Xeon(R) CPU E5-2690 0 running with $2.90$ GHz
and a high load (which can make a difference for these processors). For
$n=7$ the jobs were run in parallel on different machines and some 
parts had to be divided further in order to finish, so it is hard to give precise times.
Estimating the total running time from those parts that were run on the same machine used for $n=6$,
the total time on this type of machine should be around $10$ to $12$ CPU-years.

\bigskip

\begin{table}
\begin{tabular}{l|c|c}
$n$ & \UCS & labeled \UCS \\
\hline
 1 & 1  & 1  \\
 2 & 3  & 4  \\
 3 & 14  & 45  \\
 4 & 165  & 2.271  \\
 5 &  14.480 & 1.373.701  \\
 6 &  108.281.182 & 75.965.474.236  \\
 7 & 2.796.163.091.470.050 & 14.087.647.703.920.103.947 \\
\end{tabular}
\caption{\label{tab:ucs} The number of \UCSs and labeled \UCSsp.}
\end{table}

\bigskip

\begin{table}
\begin{tabular}{l|c|c}
$n$ & Moore families & labeled Moore families\\
\hline
 1 &  2 &  2  \\
 2 &  5 &  7  \\
 3 &  19 &  61  \\
 4 &  184 &  2480  \\
 5 &  14.664 &  1.385.552  \\
 6 & 108.295.846  &  75.973.751.474   \\
 7 & 2.796.163.199.765.896 & 14.087.648.235.707.352.472\\
\end{tabular}
\caption{\label{tab:moore} The number of Moore families and labeled Moore families.}
\end{table}

\bigskip

A \UCS on $n$ elements is called {\em sparse} if the average number of
elements in a set -- not counting the empty set -- is at most
$\frac{n}{2}$. For \UCSs that are not sparse, the \UCSs conjecture is
trivially true. The following table gives the number of sparse \UCSsp.
These numbers were computed once by filtering all \UCSs and once by
an independent implementation using special bounding criteria described in \cite{ucsthesis}.

\bigskip

\begin{table}
\begin{tabular}{l|c}
$n$ & sparse \UCS \\
\hline
 1 &  0 \\
 2 &  0 \\
 3 &  0 \\
 4 &  2 \\
 5 &  27 \\
 6 &  3.133 \\
 7 & 5.777.931 \\
\end{tabular}
\caption{\label{tab:ucssparse} The number of sparse \UCSsp.}
\end{table}

\bigskip

\subsection*{\underline{Testing}}

In \cite{ucsthesis} an independent implementation of the algorithm together
with special bounding criteria for sparse \UCSs was developed. The implementation
was used to generate all isomorphism classes 
of \UCSs for $\Omega_1, \dots ,\Omega_6$, and -- using special bounding
criteria -- to confirm the numbers of  isomorphism classes 
of sparse \UCSs  for $\Omega_7$.

A further and independent confirmation of the numbers for $\Omega_1, \dots ,\Omega_6$
and also an independent confirmation for $\Omega_7$ was obtained by computing the
number of labeled structures corresponding to each representative from the size
of the automorphism group. Note that as the size of the automorphism group is
known in the algorithm anyway, the additional costs for this test can be neglected. From this
we computed the number of (labeled) Moore families and got complete agreement
with \cite{Moorefamilies7} for $\Omega_1, \dots ,\Omega_7$. Due to the completely 
different approaches this makes implementation errors leading to the same incorrect results
in both cases extremely unlikely.

\subsection*{Acknowledgements}

We want to thank Craig Larson for introducing us to these interesting structures
and for interesting discussions.
Furthermore we want to thank Anastasiia Zharkova and Mikhail Abrosimov.
The first version of this algorithm was intended as a hands-on example to illustrate
isomorphism rejection techniques in a course they attended
at Ghent university. Without their visit, this algorithm might not have been developed.

\section*{References}
\bibliographystyle{abbrvnat}
\bibliography{/home/gbrinkma/schreib/literatur}

\end{document}